\documentclass{amsart}
\usepackage{amsfonts}
\usepackage{amsmath,amssymb}
\usepackage{amsthm}
\usepackage{amscd}
\usepackage{graphics}
\usepackage{graphicx}
\theoremstyle{remark}{
\newtheorem{Def}{{\rm Definition}}
\newtheorem{Ex}{{\rm Example}}
\newtheorem{Rem}{{\rm Remark}}
\newtheorem{Prob}{{\rm Problem}}
\newtheorem*{MainProb}{Main Problem}
}

\newtheorem{Prop}{Proposition}
\newtheorem{Thm}{Theorem}
\newtheorem{Lem}{Lemma}
\newtheorem{Fact}{Fact}
\begin{document}
\title[On homology classes of closed submanifolds in compact manifolds]{Realizing a homology class of a compact manifold by a homology class of an explicit closed submanifold--a new approach to Thom's works on homology classes of submanifolds--}
\author{Naoki Kitazawa}
\address{Institute of Mathematics for Industry, Kyushu University, 744 Motooka, Nishi-ku Fukuoka 819-0395, Japan\\
 TEL (Office): +81-92-802-4402 \\
 FAX (Office): +81-92-802-4405}
\email{n-kitazawa@imi.kyushu-u.ac.jp}
\maketitle
\begin{abstract}
It is a classical important problem of differential topology by Thom; for a homology class of a compact manifold, can we realize this by a closed (smooth) submanifold with no boundary? This is true if the degree of the class is smaller or equal to the half of the dimension of the outer manifold under the condition that the coefficient ring is $\mathbb{Z}/2\mathbb{Z}$ and that the outer manifold is closed and smooth. If the degree of the class is smaller or equal to $6$ or equal to $k-2$ or $k-1$ under the condition that the coefficient ring is $\mathbb{Z}$ where $k$ is the dimension of the manifold and that the outer manifold is closed, orientable and smooth, then this is also true. As a specific study, for $4$-dimensional closed manifolds, the topologies of closed and connected surfaces realizing given 2nd homology classes have been actively studied, for example. 

In the present paper, we consider the following similar problem; can we realize a homology class of a compact manifold by a homology class of an explicit closed manifold embedded in the (interior of the) given compact manifold? This problem is considered as a variant of previous problems. 
We present an affirmative answer via important theory in the singularity theory of differentiable maps: lifting a given smooth map to an embedding or obtaining an embedding such that the composition of this with the canonical projection is the given map. Presenting this application of lifting smooth maps and related fundamental propositions is also a main purpose of the present paper.

          

\end{abstract}


\maketitle
\section{Introduction and fundamental notation and terminologies.}
\label{sec:1}
  
The following problem, essentially launched by Thom, is a classical important problem in differential topology. We discuss the problem and related problems here in the smooth category. Let $A$ be a module.

\begin{Prob}
For a homology class $c \in H_j(X;A)$ of a closed manifold $X$, can we realize this by a closed submanifold $Y$ with no boundary or can we represent $c$ as $i_{\ast}({\nu}_Y)=c$ for a generator ${\nu}_{Y} \in H_{\dim Y}(Y;A)$ of the module $H_{\dim Y}(Y;A)$ and the inclusion $i:Y \rightarrow X$?
\end{Prob} 
This is true if the degree of the class is smaller or equal to the half of the dimension of the outer manifold under the condition that the
 coefficient $A$ is the ring $\mathbb{Z}/2\mathbb{Z}$. If the degree of the class is smaller or equal to 6 or equal to $\dim X-2$ or $\dim X-1$ under the conditions that the outer manifold $X$ is orientable and that the coefficient is the ring $\mathbb{Z}$, then this is also true. See \cite{thom} and see also \cite{suzuki} for related classical studies. Furthermore, \cite{bohrhankekotschick} is on classes we cannot represent in this way.
\begin{Prob}
For a 2nd homology class of a $4$-dimensional closed manifold, how about the orientability and the genus of a closed and connected surface realizing this?
\end{Prob} 
This gives various, explicit, important and interesting problems. See \cite{kronheimer} for example. In low dimensional differential topology, these kinds of problems are actively studied via technique on low dimensional topology, gauge theory, and so on.
\subsection{Problems studied in this paper related to these problems by Thom}

In the present paper, related to the problems before, we consider the following problem.
\begin{MainProb}
For a homology class $c \in H_j(X;A)$ of a compact manifold $X$, can we realize the homology class by a homology class of a closed manifold $Y$ satisfying $\dim Y \geq j$ and $\partial Y=\emptyset$ embedded in the (interior of the) manifold or can we represent $c$ as $i_{\ast}(c^{\prime})=c$ for
 a class $c^{\prime} \in H_{j}(Y;A)$ of the module $H_{j}(Y;A)$ which may not be of degree $\dim Y$ and the inclusion $i:Y \rightarrow X$? Moreover, can we obtain $Y$ in a constructive way?
\end{MainProb} 
Note that we consider manifolds of arbitrary dimensions in the present paper. 
\subsection{Fold maps and Reeb spaces}
\subsubsection{Fold maps}
We introduce terminologies on differentiable maps. A {\it singular} point of a differentiable map $c:X \rightarrow Y$ is a point at which the rank of the differential of the map drops or the image of the differential is smaller than both the dimensions of $X$ and $Y$. A {\it singular value} of the map is a point realized as a value at a singular point of the map. The set $S(c)$ of all singular points is the {\it singular set} of the map. The {\it singular value set} is the image $c(S(c))$ of the singular set. The {\it regular value set} of the map is the complementary set $Y-c(S(c))$ of the singular value set and a {\it regular value} is a point in the regular value set.

Throughout the present paper, manifolds are assumed to be smooth or differentiable and of $C^{\infty}$ and so are maps between manifolds. A diffeomorphism on a manifold is always smooth and the {\it diffeomorphism group} of it is defined as the group of all diffeomorphisms on it.
\begin{Def}
\label{def:1}
Let $m \geq n \geq 1$ be integers. A smooth map from an $m$-dimensional smooth manifold with no boundary into an $n$-dimensional smooth manifold with no boundary is said to be a {\it fold} map if at each singular point $p$, the form is $(x_1, \cdots, x_m)\mapsto (x_1, \cdots, x_{n-1}, {\Sigma}_{k=n}^{m-i} {x_k}^2-{\Sigma}_{k=m-i+1}^{m} {x_k}^2)$ for some coordinates and an integer $0 \leq i(p) \leq \frac{m-n+1}{2}$.
\end{Def}
For a fold map in Definition \ref{def:1}, we have the following two properties.
\begin{enumerate}
\item For any singular point $p$, the $i(p)$ is unique {\rm (}$i(p)$ is called the {\it index} of $p${\rm )}.
\item The set consisting of all singular points of a fixed index of the map is a closed submanifold of dimension $n-1$ with no boundary of the domain and the restriction to the singular set is an immersion of codimension $1$.
\end{enumerate}
A fold map is a {\it special generic} map if $i(p)=0$ for all singular points. A Morse function with exactly two singular points on a homotopy sphere and canonical projections of unit spheres are simplest examples of special generic maps.
If $m=n$, then a fold map is always special generic.
\subsubsection{Reeb spaces}
The {\it Reeb space} of a continuous map is defined as the space of all connected components of preimages of the map. 
\begin{Def}
\label{def:2}
 Let $X$ and $Y$ be topological spaces. For a continuous map $c:X \rightarrow Y$, we define a relation on $X$ as $p_1 {\sim}_c p_2$ if and only if $p_1$ and $p_2$ are in
 a same connected component of $c^{-1}(p)$ for some $p \in Y$. Thus ${\sim}_{c}$ is an equivalence relation on $X$ and we denote the quotient space $X/{\sim}_c$ by $W_c$ and call it the {\it Reeb space} of $c$.
\end{Def}

We denote the induced quotient map from $X$ into $W_c$ by $q_c$ and we can define $\bar{c}: W_c \rightarrow Y$ uniquely by $c=\bar{c} \circ q_c$. See also \cite{reeb} for example.
\begin{Prop}[\cite{shiota}]
\label{prop:1}
For a fold map, the Reeb space is a polyhedron whose dimension is equal to that of the target manifold.
\end{Prop}
For suitable classes of these maps, Reeb spaces inherit topological properties of the manifolds admitting the maps. See \cite{kitazawa}, \cite{kitazawa2}, \cite{kitazawa3} and \cite{saekisuzuoka} for example.
\subsection{Organization of the present paper}
This paper concerns Main Problem and also presents new answers to important problems on differentiable maps between manifolds.  The problems are on lifting of maps, which is a key ingredient in an explicit problem related to Main Pronlem. 
The organization of the paper is as the following.

In the next section, we introduce a class of {\it standard-spherical} fold maps. 
It is defined as a fold map such that indices of all singular points are $0$ or $1$, that preimages of regular values are points or disjoint unions of standard spheres: if the codimension of the map $f:M \rightarrow N$ is $n-m=-1$, then assume also that the domain is an orientable manifold. A {\it simple} fold map $f$ is a fold map such that ${q_f}{\mid}_{S(f)}$ is injective.  
Special generic maps form a proper subclass of these classes.

In the third section, we present examples of simple standard-spherical fold maps. After that, we consider lifting these maps to embeddings: in other words we obtain embeddings such that the compositions with the canonical projections are the given maps. This is a fundamental, important and interesting studies in the theory of singularity of differentiable maps and this also gives strong tools in obtaining main results of the present paper. We also present a new answer for a kind of these problems as Theorem \ref{thm:1}: this is also an important ingredient in main results. The last section is devoted to main results or explicit answers to Main Problem via tools and theory in the third section. 
\section{Special generic maps, standard-spherical fold maps and simple fold maps}
\label{sec:2}

Throughout this paper, let $m \geq n \geq 1$ be integers, $M$ be a closed and connected manifold of dimension $m$, $N$ be a manifold of dimension $n$ without boundary and $f:M \rightarrow N$ be a smooth map unless otherwise stated.
In addition, the structure groups of bundles such that the fibers are manifolds are assumed to be
 (subgroups of) diffeomorphism groups unless otherwise stated or these bundles are assumed to be so-called {\it smooth} bundles: {\it PL} bundles are discussed as exceptional cases. A {\it linear} bundle is a smooth bundle whose fiber is a ($k+1$)-dimensional (closed) unit disc or the $k$-dimensional unit sphere in ${\mathbb{R}}^{k+1}$ and whose structure group is a subgroup of the ($k+1$)-dimensional one $O(k+1)$ acting linearly in a canonical way. A {\it PL} bundle is a bundle whose fiber is a polyhedron and whose structure group is a group consisting of PL homeomorphisms of the fiber. 
\begin{Def}
A fold map $f:M \rightarrow N$ is said to be {\it standard-spherical} if the following properties hold (we can easily know the definitions of a {\it crossing} and a {\it normal} crossing of a smooth immersion and we omit the definitions).
\begin{enumerate}
\item The restriction map $f {\mid}_{S(f)}$ is an immersion whose crossings are normal. 
\item Indices of singular points are $0$ or $1$.
\item If $m-n=1$ and there exists a singular point of index $1$ of $f$, then $M$ is orientable.
\item Preimages of regular values are disjoint unions of points or standard spheres.
\end{enumerate}
\end{Def}
\begin{Ex}
A special generic map $f$ is a map of this class. Under the condition $m=n$ its Reeb space $W_f$ is regarded as a manifold diffeomorphic to $M$. Under the condition $m>n$ its Reeb space $W_f$ is an $n$-dimensional compact manifold we can smoothly immerse into $N$. $f$ is represented as the composition of $q_f$ with the immersion and furthermore, $q_f(S(f))=\partial W_f$ and ${q_f} {\mid}_{S(f)}$ is injective.
\end{Ex}
\begin{Def}
We say a fold map $f$ is {\it simple}
 if the restriction of $q_f$ to the singular set is injective. 
\end{Def}
Special generic maps are simple. Let $m>n$. 

Under the condition $m>n$, generally, for a simple fold map $f$, for each connected component $C$ of $q_f(S(f))$, its small regular neighborhood $N(C)$ is represented as a PL bundle over $C$ whose fiber is a closed interval or a $Y$-shaped $1$-dimensional polyhedron and the composition of the restriction of the map $q_f$ to the preimage ${q_f}^{-1}(N(C))$ of the total space $N(C)$ of the bundle with the canonical projection to $C$ is a smooth bundle. 

Under the condition $m>n$, for a simple standard-spherical map $f$, for the complementary set of the union of the regular neighborhoods $N(C)$ for all connected components $C$ of $q_f(S(f))$, the restriction of $q_f$ to the preimage gives a smooth bundle whose fiber is $S^{m-n}$. The fiber
 of the bundle $N(C)$ is Y-shaped and the bundle ${q_f}^{-1}(N(C))$ is a smooth bundle whose fiber is diffeomorphic to a manifold obtained by removing the interior of a disjoint union of three standard closed discs of dimension $m-n$ smoothly and disjointly embedded into $S^{m-n}$ for each connected component $C$ consisting of singular points of index $1$. For this, see also \cite{saeki} for example. For each connected component $C$ consisting of singular points of index $0$, the fiber of the PL bundle is a closed interval and the smooth bundle is a linear bundle and this holds for general fold maps by \cite{saeki2}.
 
Under the condition $m=n$, for each connected component $C$ of the singular set, we can take a closed tubular neighborhood $N(C) \subset M$ and the canonical projection to $C$ gives a linear bundle whose fiber is a closed interval.

\begin{Rem}
In the case of a simple standard-spherical fold map $f$, we do not assume that the restriction map $f {\mid}_{S(f)}$ is an immersion whose crossings are normal. We may assume this but the assumption is not essential.
\end{Rem}

\section{S-trivial standard-spherical fold maps and lifting these maps to embeddings.}
\label{sec:3}
\begin{Def}
\label{def:5}
In section \ref{sec:2}, if in the case $m>n$ the PL bundle $N(C)$ over $C$ and the smooth bundle ${q_f}^{-1}(N(C))$ over $C$ are trivial for each $C$, then the simple fold map $f$ is said
 to be {\it S-trivial}. If in the case $m=n$, the linear bundle $N(C)$ over $C$ is trivial for each $C$, then the fold map $f$ is said to be {\it S-trivial}.
\end{Def}

We present several examples of S-trivial standard-spherical fold maps.

\begin{Ex}
\label{ex:2}
\begin{enumerate}
\label{ex:2.1}
\item A special generic map $f:M \rightarrow N$ between equidimensional manifolds such that the following properties hold.
\begin{enumerate}
\item $M$ is orientable.
\item $S(f)$ is orientable.
\end{enumerate} 
For example, let $M$ be a homotopy sphere for example. \cite{eliashberg} implies that for any homotopy class of continuous maps between $S^m$ and $S^n$ where $m=n$, we can find a fold map satisfying the properties before. 
\item
\label{ex:2.2}
Projections of bundles whose fibers are standard spheres: they are also special generic maps having no singular point.
\item
\label{ex:2.3}
\cite{costantinothurston}, \cite{ishikawakoda} and \cite{saeki3} present S-trivial standard-spherical fold maps. Through them, we can know that 3-dimensional closed and orientable manifolds of a class (the class of {\it graph manifolds}) admit such maps into surfaces. For such a map, we can regard the Reeb space $W_f$ as the {\it shadow} of the manifold. Roughly speaking, a {\it shadow} is a polyhedron at each point of which there exists a small regular neighborhood PL homeomorphic to a regular neighborhood of a point of a Reeb space of a simple fold map into the plane as before. Moreover, the space of all points the regular neighborhoods of which are $2$-dimensional discs containing the points in the interiors is a $2$-dimensional manifold and an integer called a {\it gleam} is assigned to each connected component of this. We will review the notion in Definition \ref{def:6}. See \cite{ishikawakoda} and see also \cite{turaev} and \cite{turaev2}. See also Remark \ref{rem:1}.
\item
\label{ex:2.4}
\cite{kitazawa5} and \cite{kitazawa7} present construction of families of infinitely many such maps and families of infinitely many closed and connected manifolds admitting them starting from fundamental examples. Fundamental examples are, special generic maps and examples in \cite{kitazawa}, \cite{kitazawa2} and \cite{kitazawa4} for example and we obtain new maps via surgery operations to the maps and the manifolds.
\end{enumerate}
\end{Ex}
We introduce {\it shadows}. We omit explanations on the notions of an {\it oriented} linear bundle and its ({\it relative}) {\it Euler class}. 
\begin{Def}
\label{def:6}
Let $f$ be an S-trivial standard-spherical fold map on a 3-dimensional closed, connected and orientable manifold into an orientable surface. Orient the target surface. Define for the smooth bundle ${q_f}^{-1}(N(C))$ over $C$ a trivialization for each $C$. 
$P$ denotes each connected component of the complementary set of the union of suitable small regular neighborhoods $N(C)$ of $C$ for all connected components $C$ of $q_f(S(f))$. 
We can induce the orientation on $P$ from the target surface. We can define and let us define a trivialization for the trivial bundle ${q_f}^{-1}(P)$ over $P$ for each non-closed $P$, which is regarded as an oriented linear bundle which is trivial and whose fiber is diffeomorphic to $S^1$. We can define a relative Euler class $e_P \in H^2(P,\partial P;\mathbb{Z})$ for the oriented linear bundle over each $P$ respecting the trivializations of the bundles. The value of $e_P$ at the fundamental class of $P$, which is the generator of $H_2(P,\partial P;\mathbb{Z})$ respecting the orientation of $P$, is called the {\it gleam}. We can define a {\it shadow induced from} $f$ as the tuple of the Reeb space $W_f$, the orientation of the target surface, the family of all the trivializations of the trivial bundles, and the family of all gleams.
\end{Def}

\begin{Rem}
\label{rem:1}
In Example \ref{ex:2} (\ref{ex:2.3}) and Definition \ref{def:6}, a {\it shadow} is introduced and the definiton is in general a bit different. A {\it shadow} is defined as an object suitable for the content of the present paper here, and a {\it shadow} for the $2$-dimensional polyhedron of a wider class is in fact defined as a pair of the polyhedron and the family of integers or rational numbers of the form $\frac{k}{2}$ where $k$ is an integer assigned to all connected components of a $2$-dimensional manifold obtained as a subpolyhedron of the original polyhedron by removing the interior of a small regular neighborhood of a 1-dimensional subpolyhedron defined as the subset of all non-manifold points in the original 2-dimensional polyhedron. More precisely, a {\it shadow} is for a $3$-dimensional closed, connected and orientable manifold obtained in a natural way from a shadow introduced just before uniquely and there are in general shadows of more than one type for such a $3$-dimensional manifold. The assigned numbers are said to be {\it gleams}. Similarly, from a fold map $f$ such that the restriction map $f {\mid}_{S(f)}$ is an immersion whose crossings are normal on a $3$-dimensional closed, connected and orientable manifold into a surface, which always exists, and the Reeb space $W_f$, we can obtain a {\it shadow} similarly. Giving the rigorous definiton is left to readers or see the references.    
\end{Rem}

The following proposition is a fundamental and key proposition in the present paper. 
${\rm Emd}(X,Y)$ denotes the space of all the smooth embeddings of a manifold $X$ into a manifold $Y$ endowed with the so-called {\it $C^{\infty}$ Whitney topology} (see \cite{golubitskyguillemin}). 
\begin{Fact}[\cite{budney}]
\label{fact:1}
${\rm Emd}(X,Y)$ is $\max\{0,\min \{2\dim
Y-3\dim X-4,\dim Y-dim X-2\}\}$-connected if $\dim Y \geq \dim X+2 \geq 3$ is assumed.
\end{Fact}
See also \cite{nishioka} for Fact \ref{fact:1}. The following is for example shown in Lemma 1 of \cite{kitazawa6}.
Hereafter, we call a manifold obtained by removing the interiors of three smoothly and disjointly embedded closed $k$-dimensional discs in a $k$-dimensional standard sphere a {\it pair of pants} of dimension $k$ for $k>1$.  
\begin{Lem}
\label{lem:1}
For a Morse function $\bar{f}$ on a $k$-dimensional compact manifold satisfying $k \geq 1$ and either of the following properties, we can represent this as a composition of an embedding into ${\mathbb{R}}^{k+1}$ with a canonical projection. 
\begin{enumerate}
\item A Morse function $\bar{f}$ is a function on a closed unit disc of dimension $k$ with exactly one singular point such that the singular point is in the interior, that the preimage of one of the two extrema and the boundary coincide, and that the preimage of the remaining extremum is the one-point set consisting of the singular point. In other word, it is a so-called {\it height function}.
\item $k>1$ and a Morse function $\bar{f}$ is a function on a pair of pants of dimension $k$ satisfying the following two.
\begin{enumerate}
\item It is a function with exactly one singular point and the singular point is in the interior.
\item The preimages of the two extrema are a connected component of the boundary and the disjoint union of the remaining two connected components of the boundary, respectively.
\end{enumerate} 
\end{enumerate} 
\end{Lem}
\begin{proof}[A sketch of the proof.]
The former case follows by the definition of a singular point of index $0$ of a fold map. 
The latter case follows by the argument as the following.
\begin{enumerate}
\item Decompose the pair of pants of dimension $k$ into two $k$-dimensional compact submanifolds with corners along ($k-1$)-dimensional submanifolds: one is a disjoint union of two copies of the total space of a product bundle over a closed interval whose fiber is a standard closed disc of dimension $k-1$ and the other is diffeomorphic to the product of a closed interval and a standard closed disc of dimension $k-1$, which is so-called a $1$-handle of the Morse function. 
\item On each connected component of the disjoin union of the total spaces of the product bundles, the original function is regarded as a projection of the bundle by virtue of (a relative version of) Ehresmann's fibration theorem (\cite{ehresmann}).
\end{enumerate}
\end{proof}
\begin{Prop}[\cite{nishioka} and an extension of the result there.]
\label{prop:2}
Let $n+k \geq \max\{\frac{3m+3}{2},m+n+1\}$.
For an S-trivial standard-spherical fold map $f:M \rightarrow N$, there exists an embedding $F$ such that 
$f={\pi}_{N,k} \circ F$ where ${\pi}_{N,k}:N \times {\mathbb{R}}^k \rightarrow N$ is the canonical projection onto $N$.
\end{Prop}
\begin{proof}
For each connected component $C$ of $q_f(S(f))$, consider a small regular neighborhood $N(C)$ and the local smooth map represented as the composition of ${q_f} {\mid}_{{q_f}^{-1}(N(C))} $ with ${\bar{f}} {\mid}_{N(C)}$, we can consider an embedding so that the composition with the composition of the canonical projection $N(C) \times {\mathbb{R}}^k$ onto $N(C)$ with ${\bar{f}} {\mid}_{N(C)}$ is the original map: the assumptions that the map is S-trivial together with arguments related to so-called Thom's isotopy theorem and Lemma \ref{lem:1} and that $n+k \geq m+n+1$ enable us to do this.
We extend this over the complementary set of the union of $N(C)$ in $W_f$. The essential assumption is that the space ${\rm Emb}(S^{m-n},{\mathbb{R}}^k)$ is ($n-1$)-connected together with the relation $n+k \geq m+n+1$. We consider suitable cell decompositions of $W_f$ and the complementary set. On each $1$-cell in the complemantary set, the original map is regarded as the projection of a trivial smooth bundle and we can construct a smooth embedding as before. We can extend this as projections over $2$-cells. We can do this inductively over $k$-cells where $2 \leq k\leq n$. This completes the proof. 
\end{proof}

We explain about the fact that ${\rm Emb}(S^{m-n},{\mathbb{R}}^k)$ is ($n-1$)-connected. In Fact \ref{fact:1}, the inequality $2k-3(m-n)-4 \geq 2(\frac{3m+3}{2}-n)-3(m-n)-4=n-1$ completes the proof.

Nishioka studied only cases of special generic maps on closed, connected and orientable manifolds into Euclidean spaces satisfying the relation $m>n$. For the methods here, see also \cite{saekitakase} for example. They are on a hot topic of the singularity theory of differentiable maps: lifting a smooth map to a smooth map of a suitable class into a higher dimensional space or finding a representation of the original map by a composition of a map of the suitable class with a canonical projection. Note also that in these problems, as the author know, the target manifolds of smooth maps are Euclidean spaces.  

The following theorem presents another new explicit answer for a kind of these problems and important in a main result or Theorem \ref{thm:3} later.
\begin{Thm}
\label{thm:1}
Let $f:M \rightarrow N$ be an S-trivial standard-spherical fold map from a $3$-dimensional closed, connected and orientable manifold into a $2$-dimensional orientable manifold $N$ with no boundary such that for a shadow induced from the map {\rm (}, the Reeb space in Example \ref{ex:2} {(\rm }\ref{ex:2.3}{\rm )} or Remark \ref{rem:1}{\rm )}, all gleams are even. For the map and any integer $k \geq 3$, there exists an embedding $F:M \rightarrow N \times {\mathbb{R}}^k$ such that 
$f={\pi}_{N,k} \circ F$ where ${\pi}_{N,k}:N \times {\mathbb{R}}^k \rightarrow N$ is the canonical projection onto $N$.
\end{Thm}
\begin{proof}
We prove for $k=3$ and we can prove for general $k \geq 3$.

Around each connected component $C$ of $q_f(S(f))$ and each connected component of $W_f-{\bigcup}_C N(C)$ where we abuse notation in the proof of Proposition \ref{prop:2}, we can construct smooth map into ${\mathbb{R}}^{2+k}$ similarly and we will construct. 

First we investigate the case where each connected component $O$ of $W_f-{\bigcup}_C N(C)$ is not a closed
 surface. 
 
 The restriction of $f$ to the preimage ${q_f}^{-1}(N(C))$ is, for suitable coordinates, represented as the composition of a product map of either of the following Morse functions, which are presented in Lemma \ref{lem:1}, and the identity map ${\rm id}_{C}$ with a suitable immersion into $N$. Note that the product map is, in other words, a trivial $1$-dimensional family of the Morse function. 
\begin{enumerate}
\item A Morse function on a closed unit disc of dimension $2$ with exactly one singular point such that the singular point is in the interior, that the preimage of one of the two extrema and the boundary coincide and that the preimage of the remaining extremum is the one-point set consisting of the singular point. 
\item A Morse function on a pair of pants of dimension $2$ with exactly one singular point such that the singular point is in the interior and that the preimages of the two extrema are a connected component of the boundary and the disjoint union of the remaining two connected components of the boundary. 
\end{enumerate} 
In the situation of the proof of Proposition \ref{prop:2}, they are represented as the compositions of embeddings into $I \times {\mathbb{R}}^3$ with the canonical projection to the first component where $I$ is a closed interval, regarded as a fiber of a suitable trivial bundle over $C$. We can consider functions obtained by considering the 1-dimensional higher versions of these two functions as presented in Lemma \ref{lem:1} as the second Morse functions. They are also represented as the compositions of embeddings into $I \times {\mathbb{R}}^3$ with the canonical projection to the first component where $I$ is a closed interval, regarded as a fiber of the trivial bundle over $C$, in the same situation. Furthermore, the original functions are obtained as restrictions of the corresponding functions on the $3$-dimensional closed unit discs or the pairs of pants of dimension $3$. Moreover, the boundary of the domain of each function on the surface is obtained as a closed submanifold with no boundary of the boundary of the domain of the Morse function on the $3$-dimensional manifold so that each connected component is an equator of each of the three $2$-dimensional standard spheres in the boundary of the $3$-dimensional manifold. Distinct circles are embedded as equators in distinct $2$-dimensional spheres in the boundary of the $3$-dimensional manifold.

On each connected component $O$ of $W_f-{\bigcup}_C N(C)$, the map $f$ is regarded as the composition of the projection of a trivial linear bundle whose fiber is a circle with a suitable immersion into $N$. We can represent the projection as the composition of an embedding into $O \times {\mathbb{R}}^3$ with the canonical projection to $O$. We can also construct a map regarded as the composition of the projection of a linear bundle whose fiber is a $2$-dimensional standard sphere with a suitable immersion into $N$ and we can represent this as the composition of the composition of an embedding into $O \times {\mathbb{R}}^3$ with the canonical projection to $O$ with the last immersion. Furthermore, the original projection is obtained as the restriction of the corresponding projection of the total space of the linear bundle whose fiber is the $2$-dimensional sphere by restricting each fiber to the equator of the sphere. In other words, the original bundle is regarded as a subbundle.

We can glue the local maps on $3$-dimensional and $4$-dimensional manifolds together to obtain global embeddings into $N \times {\mathbb{R}}^{3}$ and a smooth map into $N$. Note for example that we obtain the map on the $3$-dimensional closed manifold $M$ into $N$ as the given map $f$. Connected components of the boundaries where we glue the maps together are, regarded as total spaces of trivial linear bundles over circles whose fibers are $S^1$ or $S^2$ and the restrictions of the map $q_f$ to the spaces give the projections. Note that these fibers are in a preimage of the projection ${\pi}_{N,3}:N \times {\mathbb{R}}^3 \rightarrow N$ and $S^2$ is embedded as a so-called {\it unknot} in ${\mathbb{R}}^3$ in the smooth category. $S^1$ is, as explained, embedded in $S^2 \subset {\mathbb{R}}^3$ here as an equator. The assumption on the gleams implies that the way we glue the maps yields global maps: regard each copy of $S^1 \times S^2$ as a trivial linear bundle over $S^1$ equipped with a projection regarded as a canonical projection to the first component (where the trivialization is given). The bundle isomorphisms between two of these bundles is regarded as a product map of diffeomorphisms between the base spaces, regarded as $S^1 \times \{{\ast}_1\}$ and fibers, regarded as $\{{\ast}_2\} \times S^2$.   

We explain this more precisely. We consider each pair of subbundles $S^1 \times S^1 \subset S^1 \times S^2$ whose fibers are equators to obtain the map $f$ on the $3$-dimensional manifold $M$. The assumption on the gleams implies that the bundle isomorphism is $(x,y) \mapsto r(x)(y)$ where $r:S^1 \rightarrow SO(2)$ is a smooth map and $r(x)$ acts on each fiber, diffeomorphic to $S^1$, linearly in a canonical way. The homotopy class of $r$ is $2a$ times a generator for an integer $a$. It is a fundamental argument on homotopy groups of classical Lie groups that the bundle isomorphism $(x,y) \mapsto r(x)(y)$ extend to an isomorphism $(x,y) \mapsto r_0(x)(y)$ between the two trivial linear bundles $S^1 \times S^2 \supset S^1 \times S^1$ over $S^1$ where $r_0:S^1 \rightarrow SO(3)$ is a null-homotopic smooth map and $r_0(x)$ acts on each fiber, diffeomorphic to $S^2$, linearly in a canonical way. This guarantees the argument of the global construction in the last paragraph.  

Last we investigate the case where there exists no singular point of the map $f$ and this completes the proof. We can show the statement similarly by the assumption on the gleams or the discussion on the attachments of the local maps and spaces just before (there exists exactly one connected component a gleam is assigned to). Note also that this case accounts for the case which we cannot treat in the first case.
 
This completes the proof.
\end{proof}

Exmaple \ref{ex:3} presents infinitely many examples of maps and $3$-dimensional closed, connected and orientable manifolds to which we can apply Theorem \ref{thm:1}. We omit fundamental notions on general $k$-dimensional linear bundles
such that {\it oriented} linear bundles and the {\it Euler classes} of them, defined as $k$-th integral cohomology classes. 

\begin{Ex}
\label{ex:3}
Projections of bundles whose fibers are circles over closed and orientable surfaces, which are also regarded as $2$-dimensional linear bundles such that the Euler classes are divisible by $2$ (if it is oriented), are simplest examples satisfying the assumption of Theorem \ref{thm:1}.  
\end{Ex}

\section{Main theorems}
\begin{Def}
\label{def:7}
For a homology class of a compact manifold, if for the homology group of this, the coefficient is a commutative ring $R$ having the unique identity element $1 \neq 0 \in R$ and a homology class $c \neq 0$ satisfies the following properties, then it is said to be a {\it UFG}.
\begin{enumerate}
\item If $rc=0$ for $r \in R$, then $r=0$.
\item For any element $r \in R$ which is not a unit and any homology class $c^{\prime}$, $c$ is never represented as $rc^{\prime}$.
\end{enumerate}
\end{Def}
\begin{Def}
\label{def:8}
Let $R$ be a commutative ring having the unique identity element $1 \neq 0 \in R$.
For a UFG homology class $c \in H_j(X;R)$ of a compact manifold $X$, the class $c^{\ast} \in H^j(X;R)$ satisfying the following properties is said to be the {\it dual} $c$.
\begin{enumerate}
\item $c^{\ast}(c)=1$.
\item For any submodule $B$ of $H_j(X;R)$ such that the internal direct sum of the submodule generated by $c$ and $B$ is $H_j(X;R)$, $c^{\ast}(B)=0$.
\end{enumerate}
\end{Def}
\begin{Def}
\label{def:9}
For two closed, connected, oriented and equidimensional manifolds $X$ and $Y$, we denote by ${\nu}_X \in H_{\dim X}(X;\mathbb{Z})$ and ${\nu}_Y \in H_{\dim Y}(Y;\mathbb{Z})$ the {\it fundamental classes} of the manifolds or the generators compatible with the orientations: the groups are isomorphic to $\mathbb{Z}$. The {\it mapping degree} of a continuous map $c:X \rightarrow Y$ is the unique integer $d(c)$ satisfying ${\nu}_Y=d(c) c_{\ast}({\nu}_X)$.
For two closed, connected and equidimensional manifolds $X^{\prime}$ and $Y^{\prime}$, we denote by ${\nu}_{X^{\prime}} \in H_{\dim {X^{\prime}}}(X^{\prime};\mathbb{Z}/2\mathbb{Z})$ and
${\nu}_{Y^{\prime}} \in H_{\dim Y^{\prime}}(Y^{\prime};\mathbb{Z}/2\mathbb{Z})$ the {\it $\mathbb{Z}/2\mathbb{Z}$ fundamental classes} of the manifolds or the generators of the groups, isomorphic to $\mathbb{Z}/2\mathbb{Z}$.
The {\it $\mathbb{Z}/2\mathbb{Z}$ mapping degree} of a continuous map $c^{\prime}:X^{\prime} \rightarrow Y^{\prime}$ is the unique element $d(c)=0,1 \in \mathbb{Z}/2\mathbb{Z}$ satisfying ${\nu}_{Y^{\prime}}=d(c) c_{\ast}({\nu}_{X^{\prime}})$.
\end{Def}
\begin{Def}
\label{def:10}
Let $X$ be a compact manifold. Let $Y$ be a closed manifold satisfying $\dim Y<\dim X$. A class $c_X \in H_j(X;R)$ is {\it realized} by a class $c_Y \in H_j(Y;R)$ if for an embedding $i_{Y,X}:Y \rightarrow X$ whose image
is in ${\rm Int} X$, $c_X=i_{Y,X} {\ast}(c_Y)$ holds.
\end{Def}
\begin{Thm}
\label{thm:2}
Let $R:=\mathbb{Z}, \mathbb{Z}/l\mathbb{Z}$ for an integer $l >1$. Let $X$ be a compact manifold. Let a closed and connected manifold $M$ of dimension $m \geq n$ admit an S-trivial standard-spherical fold map $f:M \rightarrow N$ and the following properties hold.
\begin{enumerate}
\item If $R$ is not isomorphic to $\mathbb{Z}/2\mathbb{Z}$, then $M$ and $N$ are orientable and $N$ is oriented so that the fundamental class is ${\nu}_N \in H_n(N;R)$. We also
use ${\nu}_N \in H_n(N;R)$ for the $\mathbb{Z}/2\mathbb{Z}$ fundamental class where $R$ is isomorphic to $\mathbb{Z}/2\mathbb{Z}$.
\item There exists a preimage $F$ of a regular value for the map $f$. Let a class $c_F \in H_{m-n}(M;R)$ be represented as the sum of the fundamental classes of all connected components of $F$ oriented canonically respecting the orientation of $N$ and an orientation of $M$ if $R$ is not isomorphic to $\mathbb{Z}/2\mathbb{Z}$. If $R$ is isomorphic to $\mathbb{Z}/2\mathbb{Z}$, then let a class $c_F \in H_{m-n}(M;R)$ be represented
as the sum of the $\mathbb{Z}/2\mathbb{Z}$ fundamental classes of all connected components of $F$.
\item There exists an element $k \in R$ and a UFG $c_{F,0} \in H_{m-n}(M;R)$ satisfying $c_F=kc_{F_0}$ and $H_{m-n}(M;R)$ is the internal direct sum of the submodule generated by the one element set $\{c_{F_0}\}$ and a suitable submodule.
\item We can take an embedding $i_{N,X}:N \rightarrow X$ satisfying $i_{N,X}(N) \subset {\rm Int}  X$, $c={i_{N,X}}_{\ast}({\nu}_N)$ and that the normal bundle of the image is trivial.
\item $\dim X \geq \max\{\frac{3m+3}{2},m+n+1\}$.
\end{enumerate}
In this situation, $kc$ is realized by the class $k{\rm PD}_{R}({c_{F,0}}^{\ast})$.
\end{Thm}
\begin{proof}
In the situation of Proposition \ref{prop:2}, we consider a map $i_{N,X,0} \circ f:M \rightarrow i_{N,X}(N)$ where $i_{N,X,0}$ is the embedding defined by restricting the target space of $i_{N,X}$ to $i_{N,X}(N)$ and take the total space of the normal bundle of the image, which is a trivial linear bundle over the image, instead of "$N \times {\mathbb{R}}^k$" in the situation of Proposition \ref{prop:2}.

$F$ is the preimage of a regular value $p$ for the map $f$ and the preimage of the regular value $i_{N,X}(p)$ for $i_{N,X,0} \circ f:M \rightarrow i_{N,X}(N)$ is diffeomorphic to $F$ and $c_F=kc_{F_0} \in H_{m-n}(M;R)$ is
also represented as the sum of all fundamental classes of all connected component of the preimage where the image $i_{N,X}(N)$ is canonically oriented from $N$ and defining the fundamental classes before respecting an orientation of
$M$ (if $R$ is not isomorphic to $\mathbb{Z}/2\mathbb{Z}$). By virtue of the first three properties on $f$, we have ${(i_{N,X,0} \circ f)}_{\ast}({\rm PD}_{R}({c_{F}}^{\ast}))=k{i_{N,X,0}}_{\ast}({\nu}_N)$. Proposition \ref{prop:2} with the last property produces a smooth embedding into the total space of a normal bundle of $i_{N,X}(N)$ in $X$: by the fourth property the bundle is trivial. Furthermore, the embedding can be taken so that the composition with the projection onto $i_{N,X}(N)$ is $i_{N,X,0} \circ f$. $kc=k{i_{N,X}}_{\ast}({\nu}_N)={i_{N \subset X}}_{\ast}(k{i_{N,X,0}}_{\ast}({\nu}_N))={i_{N \subset X}}_{\ast}({(i_{N,X,0} \circ f)}_{\ast}({\rm PD}_{R}({c_{F}}^{\ast})))$ is represented by ${\rm PD}_{R}({c_{F}}^{\ast})$ where $i_{N \subset X}$ is the inclusion into $X$ by virtue of the relation before. We have the result.
\end{proof}
Note that the last condition on the dimensions is for lifting fold maps to smooth embeddings via Proposition \ref{prop:2} and Theorem \ref{thm:1} and in Theorem \ref{thm:2} we do not consider the case $(m,n,\dim X)=(3,2,5)$ or the case of Theorem \ref{thm:1}. However, in this case this argument works. Theorem \ref{thm:3} does not exclude the last case and it is a specific case of Theorem \ref{thm:2}. Theorem \ref{thm:4} is also another specific case.
\begin{Thm}
\label{thm:3}
Let $R:=\mathbb{Z},  \mathbb{Z}/l\mathbb{Z}$ for an integer $l >1$.Let $X$ be a compact and spin manifold and let $c \in H_{2}(X;R)$ be realized by the fundamental class of a closed, connected and orientable manifold $N_0$ of dimension $n=2$ for an orientation of $N_0$. Let $M$ be a closed and connected manifold of dimension $m \geq n=2$. Let $N$ be a closed, connected and orientable manifold $N$ of dimension $n=2$.
If $R$ is not isomorphic to $\mathbb{Z}/2\mathbb{Z}$, then $M$ is assumed to be orientable and $N$ and $N_0$ are oriented so that the fundamental classes are ${\nu}_N \in H_2(N;R)$ and ${\nu}_{N_0} \in H_2(N_0;R)$, respectively. In this case, we also assume that there exists a smooth map $c_N:N \rightarrow N_0$ of mapping degree $k_2$. We also use ${\nu}_N \in H_2(N;R)$ and ${\nu}_{N_0} \in H_2(N_0;R)$ for $\mathbb{Z}/2\mathbb{Z}$ fundamental classes where $R$ is isomorphic to $\mathbb{Z}/2\mathbb{Z}$ and in this case we also assume that there exists a smooth map $c_N:N \rightarrow N_0$ of $\mathbb{Z}/2\mathbb{Z}$ of mapping degree $k_2$. Suppose that an S-trivial standard-spherical fold map $f:M \rightarrow N$ exists. We also denote the element of $R$ by $k_2$ obtained by mapping $k_2$ into $R$ in the canonical way. Assume also he following properties.
\begin{enumerate}
\item There exists a preimage $F$ of a regular value for the map $f$. Let a class $c_F \in H_{m-n}(M;R)$ be represented as the sum of the fundamental classes of all connected components oriented canonically respecting the orientation of $N$ and an orientation of $M$ if $R$ is not isomorphic to $\mathbb{Z}/2\mathbb{Z}$. If $R$ is isomorphic to $\mathbb{Z}/2\mathbb{Z}$, then let a class $c_F \in H_{m-n}(M;R)$ be represented as the sum of the $\mathbb{Z}/2\mathbb{Z}$ fundamental classes of all connected components.
\item There exists an element $k \in R$ and a UFG $c_{F,0} \in H_{m-2}(M;R)$ satisfying $c_F=kc_{F_0}$ and $H_{m-2}(M;R)$ is the internal direct sum of the submodule generated by the one element set $\{c_{F_0}\}$ and a suitable submodule.
\item We can take an embedding $i_{N_0,X}:N_0 \rightarrow X$ satisfying $i_{N_0,X}(N_0) \subset {\rm Int} X$ and $c={i_{N_0,X}}_{\ast}({\nu}_{N_0})$, which is assumed in the beginning.
\item $\dim X \geq \max\{\frac{3m+3}{2},m+3\}$. If $f$ is an S-trivial map as in Theorem \ref{thm:1} with $(m,n)=(3,2)$, then $M$ is orientable and $\dim X \geq 5$.
\end{enumerate}
In this situation, $k k_2 c$ is realized by the class $k k_2 {\rm PD}_{R}({c_{F,0}}^{\ast})$.
\end{Thm}
\begin{proof}
First from the fourth property and other assumptions on the dimensions of the manifolds, $\dim X \geq 5$ follows. 
The existence of $c_N:N \rightarrow N_0$ and $i_{N_0,X}:N_0 \rightarrow X$ in the third property yields the existence of a smooth embedding $i_{N,X}:N \rightarrow X$ of the $2$-dimensional manifold $N$ satisfying $i_{N,X}(N) \subset {\rm Int} X$ and $k_2 c={i_{N,X}}_{\ast}({\nu}_N)$. The assumptions that $X$ is spin and that $N$ and $N_0$ are closed, connected and orientable surfaces guarantee that the normal bundle of the image is trivial together with fundamental arguments on characteristic classes of linear bundles over manifolds related to smooth embeddings of smooth manifolds. In the situation of Theorem \ref{thm:2} let $n=2$ and we do not exclude the case $(m,n,\dim X)=(3,2,5)$. We can apply Theorem \ref{thm:2} together with Theorem \ref{thm:1} for $(m,n,\dim X)=(3,2,5)$.
\end{proof}
\begin{Thm}
\label{thm:4}
Let $R:=\mathbb{Z}, \mathbb{Z}/l\mathbb{Z}$ for an integer $l >1$.
Let $X$ be a compact and spin manifold and let $c \in H_{3}(X;R)$ be realized by the fundamental class of a closed, connected and orientable manifold $N_0$ of dimension $n=3$ for an orientation of $N_0$. Let $M$ be a closed and connected manifold of dimension $m \geq n=3$. Let $N$ be a closed, connected and orientable manifold $N$ of dimension $n=3$.
If $R$ is not isomorphic to $\mathbb{Z}/2\mathbb{Z}$, then $M$, $N$ and $N_0$ are orientable and $N$ and $N_0$ are oriented so that the fundamental classes are ${\nu}_N \in H_3(N;R)$ and ${\nu}_{N_0} \in H_3(N_0;R)$, respectively. In this case, we also assume that there exists a smooth map $c_N:N \rightarrow N_0$ of mapping degree $k_3$. We also use ${\nu}_N \in H_3(N;R)$ and ${\nu}_{N_0} \in H_3(N_0;R)$ for $\mathbb{Z}/2\mathbb{Z}$ fundamental classes where $R$ is isomorphic to $\mathbb{Z}/2\mathbb{Z}$ and in this case we also assume that there exists a smooth map $c_N:N \rightarrow N_0$ of $\mathbb{Z}/2\mathbb{Z}$ mapping degree $k_3$. 
Suppose that an S-trivial standard-spherical fold map $f:M \rightarrow N$ exists. We also denote the element of $R$ by $k_3$ obtained by mapping $k_3$ into $R$ in the canonical way. Assume also the following properties.
\begin{enumerate}
\item There exists a preimage $F$ of a regular value for the map $f$. Let a class $c_F \in H_{m-3}(M;R)$ be represented as the sum of the fundamental classes of all connected components oriented canonically respecting the orientation of $N$ and an orientation of $M$ if $R$ is not isomorphic to $\mathbb{Z}/2\mathbb{Z}$. If $R$ is isomorphic to $\mathbb{Z}/2\mathbb{Z}$, then let a class $c_F \in H_{m-n}(M;R)$ be represented as the sum of the $\mathbb{Z}/2\mathbb{Z}$ fundamental classes of all connected components.
\item There exists an element $k \in R$ and a UFG $c_{F,0} \in H_{m-3}(M;R)$ satisfying $c_F=kc_{F_0}$ and $H_{m-3}(M;R)$ is the internal direct sum of the submodule generated by the one element set $\{c_{F_0}\}$ and a suitable submodule.
\item We can take an embedding $i_{N_0,X}:N_0 \rightarrow X$ satisfying $i_{N_0,X}(N_0) \subset {\rm Int} X$ and $c={i_{N_0,X}}_{\ast}({\nu}_{N_0})$, which is assumed in the beginning.
\item $\dim X \geq \max\{\frac{3m+3}{2},m+4\}$.
\end{enumerate}
In this situation, $k k_3 c$ is realized by the class $k k_3 {\rm PD}_{R}({c_{F,0}}^{\ast})$.
\end{Thm}
\begin{proof}
First from the assumption and other assumptions on the dimensions of the manifolds, $\dim X \geq 7=2 \times3+1$ follows. The existence of $c_N:N \rightarrow N_0$ and $i_{N_0,X}:N_0 \rightarrow X$ in the third property yields the existence of a smooth embedding $i_{N,X}:N \rightarrow X$ of the $3$-dimensional manifold $N$ satisfying
 $i_{N,X}(N) \subset {\rm Int} X$ and $k_3 c={i_{N,X}}_{\ast}({\nu}_N)$. The assumptions that $X$ is spin, that $N$ and $N_0$ are $3$-dimensional, closed, connected and orientable and that as a result $N$ and $N_0$ are spin guarantee that the normal bundle of the image is trivial together with fundamental arguments on characteristic classes of linear bundles over manifolds related to smooth embeddings of smooth manifolds. 

In the situation of Theorem \ref{thm:2} let $n=3$.
We can apply Theorem \ref{thm:2} to complete the proof.
\end{proof}

Last we present important facts in constructing explicit cases explaining these theorems well.
\begin{enumerate}
\item For two closed, connected, oriented and equidimensional manifolds $X$ and $Y$ where $Y$ is a standard sphere of dimension larger than $0$ and an arbitrary integer $k$, there exists a smooth map $c:X \rightarrow Y$ whose mapping degree is $k$. This is important as an example for $c_N$ in Theorems \ref{thm:3} and \ref{thm:4}.
\item For two closed, connected and equidimensional manifolds $X$ and $Y$ where $Y$ is a standard sphere of dimension larger than $0$ and an arbitrary element $k \in \mathbb{Z}/2\mathbb{Z}$, there exists a smooth map $c:X \rightarrow Y$ whose $\mathbb{Z}/2\mathbb{Z}$ mapping degree is $k$. This is important as an example for $c_N$ in Theorems \ref{thm:3} and \ref{thm:4}.
\item The projections of smooth trivial bundles over closed and connected manifolds whose fibers are closed and connected manifolds present infinitely many examples for maps $f$ in Theorems \ref{thm:2}, \ref{thm:3} and \ref{thm:4}.
\item The projections of linear bundles over standard spheres which may not be trivial and whose fibers are unit spheres present infinitely many examples for maps $f$ in Theorems \ref{thm:2}, \ref{thm:3} and \ref{thm:4} for a suitable $R=\mathbb{Z}/l\mathbb{Z}$.
\end{enumerate}
In addition, Example \ref{ex:2} is also important.
\section{Acknowledgement}
\thanks{The author is a member of the project Grant-in-Aid for Scientific Research (S) (JP17H06128 Principal Investigator: Osamu Saeki)
"Innovative research of geometric topology and singularities of differentiable mappings"\\
(https://kaken.nii.ac.jp/en/grant/KAKENHI-PROJECT-17H06128/)\\ and supported by this project.}


\begin{thebibliography}{30}
\bibitem{bohrhankekotschick} C. Bohr, B. Hanke and D. Kotschick, \textsl{Cycles, submanifolds, and structures on normal bundles}, manuscripta mathematica 108 (2002), 483--494, arXiv:math/0011178.
\bibitem{budney} R. Budney, \textsl{A family of embeddings spaces}, Geometry and Topology Monographs 13 (2008), 41--83. 
\bibitem{costantinothurston} F. Costantino and D. Thurston, \textsl{$3$-manifolds efficiently bound $4$-manifolds}, J. Topol. 1 (2008),
703--745.
\bibitem{ehresmann} C. Ehresmann, \textsl{Les connexions infinitesimales dans un espace fibre differentiable}, Colloque de Topologie, Bruxelles (1950), 29--55.

\bibitem{eliashberg} Y. Eliashberg, \textsl{On singularities of folding type}, Math. USSR Izv. 4 (1970). 1119--1134.
\bibitem{golubitskyguillemin} M. Golubitsky and V. Guillemin, 
\textsl{Stable mappings and their singularities}, Graduate Texts in Mathematics (14), Springer-Verlag (1974).
\bibitem{ishikawakoda} M. Ishikawa and Y. Koda, \textsl{Stable maps and branched shadows of $3$-manifolds}, Mathematische Annalen 367 (2017), no. 3, 1819--1863, arXiv:1403.0596.
\bibitem{kitazawa} N. Kitazawa, \textsl{On round fold maps} (in Japanese), RIMS K\^{o}ky\^{u}roku Bessatsu B38 (2013), 45--59.
\bibitem{kitazawa2} N. Kitazawa, \textsl{On manifolds admitting fold maps with singular value sets of concentric spheres}, Doctoral Dissertation, Tokyo Institute of Technology (2014).
\bibitem{kitazawa3} N. Kitazawa, \textsl{Fold maps with singular value sets of concentric spheres}, Hokkaido Mathematical Journal Vol.43, No.3 (2014), 327--359.
\bibitem{kitazawa4} N. Kitazawa, \textsl{Round fold maps and the topologies and the differentiable structures of manifolds admitting explicit ones}, submitted to a refereed journal, arXiv:1304.0618 (the title has changed).
\bibitem{kitazawa5} N. Kitazawa, \textsl{Constructing fold maps by surgery operations and homological information of their Reeb spaces}, submitted to a refereed journal, arxiv:1508.05630 (the title has been changed).
\bibitem{kitazawa6} N. Kitazawa, \textsl{Lifts of spherical Morse functions}, submitted to a refereed journal, arxiv:1805.05852.
\bibitem{kitazawa7} N. Kitazawa, \textsl{Notes on fold maps obtained by surgery operations and algebraic information of their Reeb spaces}, submitted to a refereed journal, arxiv:1811.04080.
\bibitem{kronheimer} P. B. Kronheimer, \textsl{Embedded surfaces and gauge theory in three
 and four dimensions} (www.math.harvard.edu/\~{}kronheim/jdg96.pdf), Harvard University, Cambridge MA 02138.
\bibitem{nishioka} M. Nishioka, \textsl{Desingularizing special generic maps into $3$-dimensional space}, PhD Thesis (Kyushu Univ.), arxiv:1603.04520.
\bibitem{reeb} G. Reeb, \textsl{Sur les points singuliers d\`{u}ne forme de Pfaff completement integrable ou d’une fonction numerique}, -C. R. A. S. Paris 222 (1946), 847--849. 
\bibitem{saeki} O. Saeki, \textsl{Notes on the topology of folds}, J. Math. Soc. Japan Volume 44, Number 3 (1992), 551--566. 
\bibitem{saeki2} O. Saeki, \textsl{Topology of special generic maps of manifolds into Euclidean spaces}, Topology Appl. 49 (1993), 265--293.
\bibitem{saeki3} O. Saeki, \textsl{Simple stable maps of 3-manifolds into surfaces}, Topology 35 (1996), 671--698.
\bibitem{saekisuzuoka} O. Saeki and K. Suzuoka, \textsl{Generic smooth maps with sphere fibers} J. Math. Soc. Japan Volume 57, Number 3 (2005), 881--902.
\bibitem{saekitakase} O. Saeki and M. Takase, \textsl{Desingularizing special generic maps}, Journal of Gokova Geometry Topology 7 (2013), 1--24.
\bibitem{shiota} M. Shiota, \textsl{Thom's conjecture on triangulations of maps}, Topology 39 (2000), 383--399.
\bibitem{suzuki} H. Suzuki, \textsl{On the realization of homology classes by submanifolds}, Trans. Amer. Math. Soc. 87 (1958), 541--550.
\bibitem{thom} R. Thom, \textsl{Quelques propri\'et\'es globales des vari\'et\'es diff\'erentiables}, Commentarii
 mathematici Helvetici (1954), Volume 28, 17--86.
\bibitem{turaev} Vladimir G. Turaev, \textsl{Topology of shadows}, Preprint, 1991.
\bibitem{turaev2} Vladimir G. Turaev, \textsl{Shadow links and face models of statistical mechanics}, J. Differential Geom. 36 (1992), 35--74.
\end{thebibliography}
\end{document}